\documentclass[reqno]{amsart}
\numberwithin{equation}{section}

\usepackage{bbm}
\usepackage{comment}
\usepackage{verbatim}
\usepackage{mathrsfs}
\usepackage{mathtools}
\usepackage{latexsym}
\usepackage{epsfig}
\usepackage{amsmath}
\usepackage[usenames,dvipsnames]{xcolor}
\usepackage{graphics}
\usepackage[utf8]{inputenc}
\usepackage{amssymb}
\usepackage{amsthm}
\usepackage[alphabetic]{amsrefs}
\usepackage{amsopn}
\usepackage{amscd}
\usepackage[all,knot]{xy}
\xyoption{all}
\usepackage{rotating}
\usepackage{tikz}
\usepackage{hyperref}
\usepackage{tikz}
\usepackage{tikz-cd}

\theoremstyle{plain}
\newtheorem{thm}{Theorem}[section]
\newtheorem{lem}[thm]{Lemma}
\newtheorem{prop}[thm]{Proposition}
\newtheorem{cor}[thm]{Corollary}

\newtheorem*{thm*}{Theorem}
\newtheorem*{lem*}{Lemma}
\newtheorem*{prop*}{Proposition}
\newtheorem*{cor*}{Corollary}

\theoremstyle{definition}

\newtheorem*{defn*}{Definition}

\newtheorem{ex}[thm]{Example}
{}
\newtheorem{rem}[thm]{Remark}
\newtheorem*{rem*}{Remark}

{}
{}
{}
{}
{}
{}
\newtheorem{consequences}[thm]{Consequences}{}

\theoremstyle{remark}
{}
{}
{}

\def\to{\longrightarrow} 

\def\NN{\mathbb{N}}

\def\ZZ{\mathbb{Z}}
\DeclareUnicodeCharacter{221E}{$\infty$}

\def\sfD{\mathsf{D}}

\def\sfK{\mathsf{K}}

\def\mcM{\mathcal{M}}

\def\mcZ{\mathcal{Z}}

\def\mfm{\mathfrak{m}}

\def\mfp{\mathfrak{p}}

\newcommand{\mysetminusD}{\hbox{\tikz{\draw[line width=0.6pt,line cap=round] (3pt,0) -- (0,6pt);}}}
\newcommand{\mysetminusT}{\mysetminusD}
\newcommand{\mysetminusS}{\hbox{\tikz{\draw[line width=0.45pt,line cap=round] (2pt,0) -- (0,4pt);}}}
\newcommand{\mysetminusSS}{\hbox{\tikz{\draw[line width=0.4pt,line cap=round] (1.5pt,0) -- (0,3pt);}}}

\newcommand{\mysetminus}{\mathbin{\mathchoice{\mysetminusD}{\mysetminusT}{\mysetminusS}{\mysetminusSS}}}

\DeclareMathOperator{\Modu}{\mathsf{Mod}}

\DeclareMathOperator{\RHom}{\mathrm{R}Hom}

\DeclareMathOperator{\Spec}{Spec}

\DeclareMathOperator{\Spc}{Spc}

\DeclareMathOperator{\supp}{supp}

\DeclareMathOperator{\loc}{\mathrm{loc}}

\usepackage{microtype}

\usepackage{todonotes}

\definecolor{internationalkleinblue}{rgb}{0.0, 0.18, 0.65}

\title{The importance of being isolated}

\thanks{JOG was supported by the Deutsche Forschungsgemeinschaft (Project-ID 491392403 – TRR 358). GS was supported by the Danmarks Frie Forskningsfond (grant ID: 10.46540/4283-00116B)}

\author{Scott Balchin}
\address{Scott Balchin, Mathematical Sciences Research Centre, Queen’s University Belfast, UK}
\email{s.balchin@qub.ac.uk}
\urladdr{http://bifibrant.com/}

\author{Juan Omar G\'omez}
\address{Juan Omar G\'omez, Fakultat f\"ur Mathematik, Universit\"at Bielefeld, D-33501 Bielefeld, Germany}
\email{jgomez@math.uni-bielefeld.de}
\urladdr{https://sites.google.com/cimat.mx/juanomargomez/home}

\author{Greg Stevenson}
\address{Greg Stevenson, Aarhus University, Department of Mathematics, Ny Munkegade 118, bldg. 1530
DK-8000 Aarhus C, Denmark
}
\email{greg@math.au.dk}

\keywords{}

\begin{document}

\begin{abstract}
\noindent We give both a sufficient condition for and an obstruction to the derived category of a commutative ring being generated by its residue fields. As an illustration, we exhibit a ring for which Foxby’s small support classifies localizing subcategories despite the failure of the local-to-global principle. We also conclude that the residue fields do not generate for a polynomial ring in infinitely many variables. Finally, we give a necessary and sufficient criterion for the derived category to satisfy the local-to-global principle; it turns out to depend solely on the topology of the spectrum.
\end{abstract}

\maketitle




\section{Introduction}

A natural approach to understanding the unbounded derived category of a commutative ring is through its various localizations. In the situations we understand well, for instance in the noetherian setting, these derived localizations reflect the prime ideal structure of the ring. However, the collection of all localizations remains mysterious beyond the noetherian case, with only a few exceptions. In fact, we do not even know whether this class forms a set or a proper class. This lack of understanding is perhaps not surprising given the intricate nature of these localizations in the examples we have (cf. \cites{Nee00, Absflat}).

An obvious starting point is to understand the boundary along which the methods from the noetherian case break down. This leads us to the question: for which commutative rings is the unbounded derived category generated by its residue fields? After all, this is the crucial condition leading to Neeman's classification in the noetherian case \cite{Nee92}. 

On the positive side, we give a sufficient condition for the residue fields to generate:

\begin{thm*}[\ref{thm:suff}]
    Let $A$ be a commutative ring. If $\Modu A$ has Gabriel dimension then $\sfD(A)$ is generated by the residue fields. In particular, the lattice of localizing subcategories of $\sfD(A)$ is given by the power set of $\Spec A$.
\end{thm*}

Using this result, in Section \ref{sec:ex} we exhibit a ring $A$ such that the Hochster dual $(\Spec A)^\vee$ of the spectrum does not satisfy the $T_D$ separation axiom (i.e. not all points are locally closed), yet the residue fields still generate. In particular, this shows that Foxby’s small support can classify localizing subcategories even in the absence of the local-to-global principle. This also gives a negative answer to \cite{No}*{Questions~12.2 and 12.5}.

We do not know if having Gabriel dimension is a necessary condition for generation by the residue fields. In fact, we are lacking general tools one could reach for to check that the residue fields fail to generate for a given ring. To move towards remedying this, we give an obstruction, in terms of the constructible topology on the spectrum, to generation by the residue fields:

\begin{thm*}[\ref{thm:main}]
    Let $A$ be a commutative ring. If $(\Spec A)^\mathrm{con}$ does not have Cantor-Bendixson rank, then the residue fields do not generate $\sfD(A)$.
\end{thm*}

As a concrete example, we use this criterion to show that for the polynomial ring on a countably infinite set of variables the residue fields do not generate; see Example \ref{ex:1}. We emphasize that the converse of our criterion does not necessarily hold; see Example \ref{ex:2}.

Finally, we turn to the local-to-global principle. If $(\Spec A)^\vee$ satisfies the $T_D$ separation axiom then generation by the residue fields implies the local-to-global principle for $\sfD(A)$. We show that if  $(\Spec A)^\mathrm{con}$ does not have Cantor-Bendixson rank, then the local-to-global principle must fail for $\sfD(A)$. Using this observation we prove:

\begin{thm*}[\ref{thm:main2}]
    Let $A$ be a commutative ring. Then $\sfD(A)$ satisfies the local-to-global principle if and only if $(\Spec A)^\vee$ has Cantor-Bendixson rank.
\end{thm*}

As such, at least for the derived category of a commutative ring, the local-to-global principle is completely determined by the topology of the space $\Spc \sfD^\mathrm{perf}(A)$.

\subsection*{Notation}
Throughout we let $A$ denote a commutative ring. We denote by $\sfD(A)$ the unbounded derived category of $A$. All functors are derived unless explicitly mentioned otherwise; we don't decorate our functors to indicate this fact. For a set of objects $X_\lambda \in \sfD(A)$, indexed by $\Lambda$, we denote by $\loc_A(X_\lambda \mid \lambda \in \Lambda)$ the localizing subcategory they generate. We drop the subscript when the ring is clear from context. We say that \textit{the residue fields generate $\sfD(A)$} if $\loc_A(k(\mfp) \mid \mfp\in \Spec A) =\sfD(A)$. 

We refer to spaces of the form $\Spec A$, for $A$ a commutative ring, as coherent spaces (oft called spectral in the tt literature). For a coherent space $X$, we write $X^\vee$ to denote its Hochster dual.



\section{A sufficient condition}

 We begin by proving a sufficient condition for the residue fields of $A$ to generate $\sfD(A)$. We refer to \cite{GR} and \cite{Pop}*{Chapter~5.5} for the relevant terminology and preliminaries on Gabriel dimension (which is originally due to Gabriel \cite{Gabriel}). We follow the convention that $(\Modu A)_0$ consists of the semiartinian modules and let $T_\alpha \colon \Modu A \to \Modu A /(\Modu A)_\alpha$ denote the localization killing the objects at stage $\alpha$.

 \begin{thm}\label{thm:suff}
 If $\Modu A$ has Gabriel dimension then $\sfD(A)$ is generated by the residue fields. In particular, the lattice of localizing subcategories of $\sfD(A)$ is given by the power set of $\Spec A$.
 \end{thm}

The first step is to identify the residue fields as Gabriel simple modules, i.e.\ to check for each $\mfp$ that there is a successor ordinal $\gamma$ with $T_{\gamma-1}k(\mfp)$ simple and $k(\mfp)$ has no non-trivial subobject in $(\Modu A)_{\gamma-1}$. We will say that a collection of modules $\mcM$ is a representative collection of Gabriel simples if it consists of Gabriel simple modules and for each $\alpha$ every simple object of $\Modu A /(\Modu A)_\alpha$ is of the form $T_\alpha M$ for some $M\in \mcM$.

 \begin{lem}\label{lem:gsimple}
 If $A$ is a commutative ring such that $\Modu A$ has Gabriel dimension then the residue fields $k(\mfp)$ for $\mfp\in \Spec A$ are a representative collection of Gabriel simples.
 \end{lem}
 \begin{proof}
 Suppose that $E$ is an indecomposable injective $A$-module. By \cite{Pop}*{Theorem~5.16} $E$ has an associated prime, say $\mfp$, and hence we must have $E \cong E(A/\mfp)$, the injective envelope of $A/\mfp$. It is a consequence of the work of Gordon and Robson \cite{GR}*{Section~3} that the $A/\mfp$ are Gabriel simple. Then combining \cite{Gabriel}*{Proposition~IV.1.2} (see also the discussion beforehand) with our observation about the injectives implies that the $A/\mfp$ are a representative collection of Gabriel simples.

 It then just remains to notice that $k(\mfp)$ is Gabriel simple. If $T_{\gamma-1}A/\mfp$ is simple then, by virtue of $k(\mfp)$ being the filtered colimit along a sequence of endomorphisms of $A/\mfp$, the localization $T_{\gamma -1}$ sends $A/\mfp \to k(\mfp)$ to an isomorphism. Moreover, $k(\mfp)$ clearly has no non-zero subobject in $(\Modu A)_{\gamma-1}$: indeed, $A/\mfp$ already has this property and $k(\mfp)$ is an essential extension of $A/\mfp$. 
 \end{proof}


With this in hand we can prove the theorem:

\begin{proof}[Proof of Theorem \ref{thm:suff}]
We will show by transfinite induction that each $(\Modu A)_\alpha$ is contained in the localizing subcategory $\sfK = \loc(k(\mfp) \mid \mfp \in \Spec A)$ of $\sfD(A)$ generated by the residue fields. The base case for the induction, that $\alpha = 0$, is more or less immediate:  every object of $(\Modu A)_0$ has a filtration by simple modules, which are precisely the residue fields at maximal ideals, and such a filtration witnesses membership in $\sfK$.

The case of limit ordinals is also straightforward. If $\lambda$ is a limit ordinal then $(\Modu A)_\lambda$ is the closure of $\cup_{\kappa < \lambda} (\Modu A)_\kappa$ under filtered colimits. So if $M\in (\Modu A)_\lambda$ then it is a filtered colimit of modules lying in $\sfK$ (by the induction hypothesis) and hence in $\sfK$ by closure under homotopy colimits.
 
It remains to treat the case of successor ordinals, so suppose that $(\Modu A)_\alpha \subseteq \sfK$ and consider $(\Modu A)_{\alpha+1}$ which sits in the localization sequence
\[
(\Modu A)_{\alpha} \to (\Modu A)_{\alpha+1} \stackrel{T_\alpha}{\to} (\Modu A)_{\alpha+1}/(\Modu A)_{\alpha}
\]
where the rightmost category is semiartinian. As used above, by \cite{Pop}*{Theorem~5.16} every non-zero module has an associated prime and so the $A/\mfp$'s which lie in $(\Modu A)_{\alpha+1}$ filter every object in it. Thus, using the induction hypothesis, it is enough to show that the $A/\mfp$ which are $(\alpha+1)$-simple lie in $\sfK$. So say $A/\mfp$ is $(\alpha+1)$-simple. By Lemma~\ref{lem:gsimple} we know that, considering the short exact sequence,
\[
0 \to A/\mfp \to k(\mfp) \to Q \to 0,
\]
we have $T_\alpha(Q) = 0$. Hence $Q\in (\Modu A)_{\alpha}$ and this short exact sequence witnesses that $A/\mfp \in \sfK$. Thus $(\Modu A)_{\alpha+1} \subseteq \sfK$ as required.

This shows that $\Modu A \subseteq \sfK$ as $A$ has Gabriel dimension and hence $\sfK = \sfD(A)$.    
\end{proof}

 \begin{rem}
 We know the converse of Theorem~\ref{thm:suff} holds in at least one case: if the Krull dimension of $A$ is $0$ then the residue fields generate $\sfD(A)$ if and only if $A$ is semiartinian (i.e.\ it has Gabriel dimension $0$). Indeed, if the residue fields generate $\sfD(A)$ then every injective $A$-module receives a non-zero map from a simple module and hence has non-zero socle. Any module is an essential submodule of an injective module and so also has non-zero socle and thus $A$ is semiartinian (cf.\ Theorem~\ref{thm:aux}).
 \end{rem}


\section{An example}\label{sec:ex}

In this section we present an amusing example. 
Consider the commutative ring
\[
A = \frac{k[x_1, x_2, \ldots]_{(x_1,x_2,\ldots)}}{(x_ix_j \mid i\neq j)}
\]
that is, we take the local ring at the origin of union of the axes in infinite dimensional affine space. This is a local ring and is isomorphic to the ring obtained by taking $\oplus_i k[x_i]_{(x_i)}$, viewed as a non-unital ring, and freely making it a unital $k$-algebra. 

Let us compute $\Spec A$. If $\mfp$ is a prime ideal then the relations $x_ix_j = 0$ mean that there is at most one $i$ with $x_i \notin \mfp$. Let us denote by $\mfp_i$ the ideal $(x_j \mid j\neq i)$. The corresponding quotient is $A/\mfp_i = k[x_i]_{(x_i)}$ and so $\mfp_i$ is prime and it must be a minimal prime by the above discussion. The description of $A/\mfp_i$ tells us that $\mfm = (x_i \mid i\geq 1)$ is the unique prime ideal lying over $\mfp_i$ and so as a set
\[
\Spec A = \{\mfp_i \mid i\geq 1\} \cup \{\mfm\}.
\]
The basic open affine $D(x_i)$ is precisely $\{\mfp_i\}$ and so each of the minimal primes is an isolated point and one easily checks that $\Spec A$ is homeomorphic to $(\Spec \ZZ)^\vee$. In particular, $\Spec A$ is a $T_D$-space and the Hochster dual of $\Spec A$ is \emph{not} $T_D$. Thus the local-to-global principle does not hold for $\sfD(A)$ as we cannot define an idempotent at $\mfm$ using finite localizations (cf.\ Consequences~\ref{consequences}).

Let us now show that $\Modu A$ has Gabriel dimension. Let $(\Modu A)_0$ denote the colimit closed Serre subcategory of semiartinian modules and $T_0\colon \Modu A \to \Modu A / (\Modu A)_0$ the corresponding localization. The short exact sequence
\[
0 \to \mfm  \to A \to k \to 0
\]
implies that $T_0\mfm \cong T_0A$. The maximal ideal decomposes as $\mfm = \oplus_i (x_i)$ and so to show $(\Modu A)_1 = \Modu A$ it is enough to verify that the residue fields $k(\mfp_i) \cong k(x_i)$ are simple in $\Modu A / (\Modu A)_0$ and $T_0k(\mfp_i) \cong T_0(x_i)$.

We start with the latter. The short exact sequence
\[
0 \to k[x_i]_{(x_i)} \to k(\mfp_i) \to k(\mfp_i)/k[x_i]_{(x_i)} \to 0
\]
shows that $T_0k[x_i]_{(x_i)} \cong T_0k(\mfp_i)$ as the quotient is semiartinian (in fact this is just a short exact sequence of $k[x_i]_{(x_i)}$-modules and so one is doing the computation over a DVR). It is obvious that $T_0k[x_i]_{(x_i)} \cong T_0(x_i)$. From all of this we learn that killing $(\Modu A)_0$ inverts the action of $x_i$ on $T_0k[x_i]_{(x_i)}$ and so it is simple (or one can again use that we are really doing the computation over a DVR where it is standard).

\begin{rem}
The above makes explicit how the residue fields generate $\sfD(A)$.
\end{rem}

\begin{consequences}\label{consequences}
We see that $A$ can have Gabriel dimension, and hence the lattice of localizing subcategories of $\sfD(A)$ can be naturally lattice isomorphic to the power set of $\Spec A$, without $(\Spec A)^\vee$ being $T_D$. In particular, Foxby's small support can classify localizing subcategories without the local-to-global principle holding. Indeed, in the given example the punctured spectrum is not quasi-compact and so no idempotent cutting out $\{\mfm\}$ can be constructed from finite localizations. Moreover, the situation does not improve if one considers smashing localizations: one easily checks that $T_0A \cong T_0\mfm$ is not compact in $\sfD(A)/\loc(k)$ and so $\loc(k)$ cannot be smashing (in fact, the telescope conjecture holds for $\sfD(A)$). 
\end{consequences}



\section{Recollections on absolutely flat approximations}

We denote by $(\Spec A)^\mathrm{con}$ the space obtained by equipping $\Spec A$ with the constructible topology (see \cite{book_spectralspaces}*{Definition 1.3.11}). Let $f\colon A \to A^\mathrm{abs}$ be the absolutely flat approximation of $A$. This is the initial map from $A$ to a ring of weak global dimension $0$. The map $f$ induces a continuous bijection $\Spec A^\mathrm{abs} \to \Spec A$ identifying $\Spec A^\mathrm{abs}$ with $(\Spec A)^\mathrm{con}$. We regard this as an identification and use the usual symbols for points, e.g.\ $x\in \Spec A$, to refer to points in both $\Spec A$ and $\Spec A^\mathrm{abs}$ with the understanding that these correspond to different prime ideals in different rings. This is somewhat harmless as the residue fields of $A$ and $A^\mathrm{abs}$ at the point $x$ agree. 
An important original reference for material on absolutely flat rings is \cite{Olivier_UAF} and some relevant material is collected in \cite{Absflat}.

For an object $X \in \sfD(A^\mathrm{abs})$ we set
\[
\supp_{A^\mathrm{abs}} X = \{x\in \Spec A^\mathrm{abs} \mid X \otimes k(x) \neq 0 \} = \{x\in \Spec A^\mathrm{abs} \mid X_x \neq 0 \}
\]
where the second equality holds since $k(x)\simeq (A^\mathrm{abs})_x$.  Note that $\supp_{A^\mathrm{abs}} X = \varnothing$ if and only if $X\cong 0$, which is immediate from the second description. 

We will need the following fact which combines \cite{Absflat}*{Theorem~4.7} and \cite{Stevensonltg}*{Theorem~6.3}.

\begin{thm}\label{thm:aux}
Let $A$ be a commutative absolutely flat ring. Then the residue fields generate $\sfD(A)$ if and only if $\Spec A$ has Cantor-Bendixson rank which occurs if and only if $A$ is semiartinian.
\end{thm}
\begin{proof}
The theorems quoted above show that the residue fields generate $A$ if and only if $A$ is semiartinian. One sees that $A$ is semiartinian if and only if $\Spec A$ has Cantor-Bendixson rank by combining \cite{Prest}*{Theorems 7.3.20 and 8.2.92}.
\end{proof}






\section{The Obstruction}

As above let $f\colon A \to A^\mathrm{abs}$ be the absolutely flat approximation of $A$. We start by computing enough about the $f^*k(x)$ for $x\in \Spec A$ to reduce the problem of generation to one in $A^\mathrm{abs}$. Here, $f^*$ denotes the derived base change functor along $f$, and we write $f_\ast$ for its right adjoint, given by restriction of scalars.

\begin{lem}\label{lem:1}
For a point $x\in \Spec A$ we have $\supp_{A^\mathrm{abs}} f^*k(x) = \{x\}$.
\end{lem}
\begin{proof}
If $y\neq x$ is another point then, by the projection formula and agreement of the residue fields for $A$ and $A^\mathrm{abs}$, we have
\begin{align*}
0 &= k(x) \otimes k(y) \\
&= k(x) \otimes f_*k(y) \\
&= f_*(f^*k(x) \otimes k(y)).
\end{align*}
The functor $f_*$ is conservative and so we deduce that $f^*k(x) \otimes k(y)$ vanishes for all $y\neq x$ and so it has support contained in $\{x\}$. The support cannot be empty as
\[
0 \neq \RHom_A(k(x), k(x)) = \RHom_A(k(x), f_*k(x)) \cong \RHom_{A^\mathrm{abs}}(f^*k(x), k(x))
\]
shows that $f^*k(x) \neq 0$ and the support on $\sfD(A^\mathrm{abs})$ detects vanishing.
\end{proof}

\begin{lem}\label{lem:2}
We have $f^*k(x) \in \loc_{A^\mathrm{abs}}(k(x))$.
\end{lem}
\begin{proof}
By Lemma~\ref{lem:1} we know $\supp_{A^\mathrm{abs}} f^*k(x) = \{x\}$. This is equivalent to the vanishing of the stalks $(f^*k(x))_y$ for $y\neq x$. Consider then the localization triangle
\[
\Gamma_{\mcZ(x)}A^\mathrm{abs} \to A^\mathrm{abs} \to k(x) \to
\]
corresponding to the Thomason (i.e.\ open) subset $\mcZ(x) = \Spec A^\mathrm{abs} \mysetminus \{x\}$. Here we identify $k(x)$ with $L_{\mcZ(x)}A^\mathrm{abs}$; see \cite[Lemma 4.2]{Absflat}. We know that
\[
0 = \Gamma_{\mcZ(x)}A^\mathrm{abs} \otimes k(x) \cong (\Gamma_{\mcZ(x)}A^\mathrm{abs})_x
\]
and thus $(f^*k(x) \otimes \Gamma_{\mcZ(x)}A^\mathrm{abs})_z \cong (f^*k(x))_z \otimes (\Gamma_{\mcZ(x)}A^\mathrm{abs})_z$ vanishes for all $z\in \Spec A^\mathrm{abs}$. Indeed, $(f^*k(x))_z = 0$ if $z\neq x$ and we have just observed that in the remaining case $z=x$ the object $(\Gamma_{\mcZ(x)}A^\mathrm{abs})_z$ vanishes. It follows that $f^*k(x) \otimes \Gamma_{\mcZ(x)}A^\mathrm{abs}$ is already trivial and so we have $f^*k(x) \cong f^*k(x) \otimes k(x)$. Thus $f^*k(x)$ lies in $\loc_{A^\mathrm{abs}}(k(x))$ as claimed.
\end{proof}

\begin{thm}\label{thm:main}
If $(\Spec A)^\mathrm{con}$ does not have Cantor-Bendixson rank then the residue fields do not generate $\sfD(A)$, that is
\[
\loc_A(k(x) \mid x\in \Spec A) \subsetneq \sfD(A).
\]
\end{thm}
\begin{proof}
We consider $f\colon A \to A^\mathrm{abs}$ as above. The functor $f_*$ is conservative and so $f^*$ preserves generation (this is well known, see \cite{JS}*{Lemma~3.3} for a precise statement). So if we have
\[
\loc_A(k(x) \mid x\in \Spec A) = \sfD(A)
\]
it would follow that
\[
\loc_{A^\mathrm{abs}}(f^*k(x) \mid x\in \Spec A) = \sfD(A^\mathrm{abs}).
\]
By Lemma~\ref{lem:2} this would imply that 
\[
\loc_{A^\mathrm{abs}}(k(x) \mid x\in \Spec A) = \sfD(A^\mathrm{abs}).
\]
However, we know from Theorem~\ref{thm:aux} that this is only possible if $\Spec A^\mathrm{abs} = (\Spec A)^\mathrm{con}$ has Cantor-Bendixson rank. So if this is not the case then the residue fields cannot generate $\sfD(A)$.
\end{proof}

The same idea obstructs the local-to-global principle.

\begin{cor}\label{cor:main}
If the Hochster dual of $\Spec A$ is $T_D$, so that the idempotents $\Gamma_xA$ exist for all $x\in \Spec A$, then if $(\Spec A)^\mathrm{con}$ does not have Cantor-Bendixson rank the local-to-global principle fails for $\sfD(A)$. 
\end{cor}
\begin{proof}
By \cite{Absflat}*{Remark~2.10} in this case $f^*\Gamma_xA \cong \Gamma_xA^\mathrm{abs} = k(x)$. So as in the argument above if the local-to-global principle held it would imply the residue fields generate for $A^\mathrm{abs}$ and hence $(\Spec A)^\mathrm{con}$ would have to have Cantor-Bendixson rank.
\end{proof}

This turns out to be the missing piece in completely determining when $\sfD(A)$ satisfies the local-to-global principle. We return to this in Section~\ref{sec:gremlins} after giving some examples.



\section{Examples}

\begin{ex}\label{ex:1}
Let $k$ be a field and $R = k[x_1,x_2,\ldots]$. We will apply Theorem~\ref{thm:main} to check that the residue fields don't generate $\sfD(R)$. Consider the ideal $I = (x_i^2 - x_i \mid i\geq 1)$ and set $S = R/I \cong k[e_1,e_2,\ldots]$ where each $e_i = e_i^2$. It is enough to show that the closed subset $\Spec S$ of $\Spec R$ does not have Cantor-Bendixson rank; see for instance, \cite[Lemma 5.2]{Stevensonltg}. 

If $\mfp \in \Spec S$ then as $e_i(1-e_i) = 0$ we must have either $e_i$ or $1-e_i$ in $\mfp$. It follows that $\mfp$ is uniquely determined by the element of $\{0,1\}^\NN$ which describes which $e_i$ are in $\mfp$. We see that $S$ is zero dimensional and since it is reduced it is absolutely flat. Thus $\Spec S$ is a compact Hausdorff space. One guesses from the above, and easily checks, that it is $\{0,1\}^\NN$ with the product topology i.e.\ the Cantor set. Hence $\Spec S$ does not have Cantor-Bendixson rank and so neither does $\Spec R$.
\end{ex}

\begin{ex}\label{ex:2}
The condition that $(\Spec A)^\mathrm{con}$ has Cantor-Bendixson rank is not sufficient for the residue fields to generate $\sfD(A)$. For instance:
\begin{enumerate}
    \item If $A$ is a non-noetherian rank one valuation domain, then the residue fields do not generate. This is a consequence of the results of \cite{BazzoniStovicek}. 
    \item Let $A = k[x_2, x_3, x_4, \ldots]/(x_2^2, x_3^3, x_4^4, \ldots)$. In this case, there is a unique prime ideal, and the results of \cite{Nee00} show that the corresponding residue field does not generate $\sfD(A)$.
\end{enumerate}
\end{ex}

\begin{rem}
    We note that none of the above examples have Gabriel dimension.
\end{rem}

\begin{ex}
  For a coherent space $X$ the constructible topology on $X$ is self-dual in the sense that $X^\mathrm{con} = (X^\vee)^\mathrm{con}$. One deduces that for each coherent space $X$ such that $X^\mathrm{con}$ does not have Cantor-Bendixson rank the classes of commutative rings
  \[
  \{A \in \mathsf{CRing} \mid \Spec A \cong X\} \text{ and } \{A \in \mathsf{CRing} \mid \Spec A \cong X^\vee\}
  \]
  consist of rings for which the residue fields don't generate the derived category.
\end{ex}


\section{Consequences for the local-to-global principle}\label{sec:gremlins}


As alluded to after Corollary~\ref{cor:main}, we are now in a position to completely characterize when the local-to-global principle holds for the derived category. In order to get there we need to discuss a little point-set topology.

Recall that a topological space $X$ is \textit{scattered} if every non-empty subset $Y \subseteq X$ contains a point isolated in $Y$. It is a classical result that $X$ being scattered is equivalent to $X$ having Cantor-Bendixson rank.

\begin{lem}
    Let $X$ be a scattered space. Then $X$ is $T_D$.
\end{lem}

\begin{proof}
    Let $x \in X$. Then $\overline{\{x\}}$ has an isolated point. Since open sets are generalization closed, this isolated point must be $x$. That is, $\{x\} =  \overline{\{x\}} \cap U$ with $U \subseteq X$ open. Hence $\{x\}$ is locally closed as required.
\end{proof}

\begin{lem}\label{lem:tdopens}
    Let $X$ be a $T_D$ space. Then for all $x \in X$ we have $(X \setminus  \overline{\{x\}}) \cup \{x\}$ is an open subset of $X$.
\end{lem}

\begin{proof}
    Choose an open $U \ni x$ such that $\{x\} = U \cap  \overline{\{x\}}$. Then
    \[
    (X \setminus  \overline{\{x\}}) \cup \{x\} = (X \setminus \overline{\{x\}}) \cup U.\qedhere
    \]
\end{proof}

\begin{lem}
    A space $X$ is scattered if and only if every non-empty closed subset of $X$ has an isolated point.
\end{lem}

\begin{proof}
    The forward direction is immediate. For the converse, suppose that closed subsets have isolated points, and let $Y \subseteq X$. Then $\overline{Y}$ has an isolated point, $y \in \overline{Y}$, i.e.\ there exists $U$ open with $U \cap \overline{Y} = \{y\}$. But then $y \in Y$ as required.
\end{proof}

\begin{lem}\label{lem:finer}
    Let $X$ be a coherent space. If $X$ is scattered then $X^{\mathrm{con}}$ is scattered.
\end{lem}

\begin{proof}
    This is clear as the constructible topology is finer than the topology of $X$.
\end{proof}

\begin{lem}\label{lem:gibisolated}
    Let $X$ be a $T_D$ coherent space such that $X^{\mathrm{con}}$ is scattered. Then $X$ has an isolated point.
\end{lem}

\begin{proof}
    As $X^{\mathrm{con}}$ is scattered, the subset $X^{\mathrm{min}}$ of generic points of $X$ has an isolated point $x$ as a subspace of $X^{\mathrm{con}}$. By \cite[Corollary 4.4.6]{book_spectralspaces}, $X^\mathrm{min}$ inherits the same topology from $X$ as it does from $X^\mathrm{con}$. As such $x$ is isolated in $X^\mathrm{min}$ as a subspace of $X$. That is, there exists an open $U \subseteq X$ such that $U \cap X^\mathrm{min} = \{x\}$. This implies $U \subseteq \overline{\{x\}}$ by generalization closure of open subsets, and the assumption that $x$ is a generic point.

    As $X$ is $T_D$, we have that $W \coloneq (X \setminus\overline{\{x\}}) \cup \{x\}$ is open by Lemma~\ref{lem:tdopens}. Then $U \cap W = \{x\}$ is isolated as required.
\end{proof}

\begin{prop}\label{prop:topology}
    Let $X$ be a $T_D$ coherent space. Then $X$ is scattered if and only if $X^{\mathrm{cons}}$ is.
\end{prop}

\begin{proof}
    The forward direction is Lemma~\ref{lem:finer}. For the converse, let $Y \subseteq X$ be a closed subset. Then $Y$ is again coherent and $T_D$, and $Y^{\mathrm{con}} \subseteq X^{\mathrm{con}}$ is scattered, so by Lemma~\ref{lem:gibisolated} $Y$ has an isolated point.
\end{proof}

\begin{thm}\label{thm:main2}
    Let $A$ be a commutative ring. Then $\mathsf{D}(A)$ satisfies the local-to-global principle if and only if $(\Spec A)^\vee$ is scattered.
\end{thm}

\begin{proof}
    If $(\Spec A)^\vee$ is scattered then it is $T_D$ and the local-to-global principle for the derived category follows from \cite[Theorem 7.18]{sanders_vanishing} (see also \cite[Theorem 4.8.9]{StevensonNotes}).

    Conversely, suppose that $\mathsf{D}(A)$ satisfies the local-to-global principle, so in particular $(\Spec A)^\vee$ is $T_D$. Then by Corollary~\ref{cor:main} we know $(\mathrm{Spec} A)^\mathrm{con}$ must have Cantor-Bendixson rank; equivalently, it is scattered. By Proposition~\ref{prop:topology} $(\Spec A)^\vee$ is then scattered.
\end{proof}



\end{document}